\documentclass[12pt]{amsart}
 \usepackage{amsfonts,amssymb}
\usepackage{color}
\usepackage{mathrsfs}
\let\mathcal\mathscr

\usepackage[all,ps,cmtip]{xy}

\def\Z{{\bf Z}}

\def\C{{\bf C}}

\def\Q{{\bf Q}}
\def\P{{\bf P}}

\def\cF{\mathcal{F}}
\def\cL{\mathcal{L}}
\def\cO{\mathcal{O}}
\def\cG{\mathcal{G}}
\def\cP{\mathcal{P}}
\def\cH{\mathcal{H}}
\def\cE{\mathcal{E}}
\def\cS{\mathcal{S}}

\def\cM{\mathcal{M}}

\def\cK{\mathcal{K}}

\def\cQ{\mathcal{Q}}
\def\cY{\mathcal{Y}}
\def\cX{\mathcal{X}}

\def\lra{\longrightarrow}
\def\llra{\hbox to 10mm{\rightarrowfill}}
\def\lllra{\hbox to 15mm{\rightarrowfill}}

\def\llla{\hbox to 10mm{\leftarrowfill}}
\def\lllla{\hbox to 15mm{\leftarrowfill}}

\def\dra{\dashrightarrow}
\def\isom{\simeq}
\def\eps{\varepsilon}
\def\ie{\hbox{i.e.}}

\def\Im{\mathop{\rm Im}\nolimits}

\DeclareMathOperator{\rk}{rk}

\DeclareMathOperator{\Ker}{Ker}

\DeclareMathOperator{\Pic}{Pic}

\DeclareMathOperator{\Hom}{Hom}
\def\Im{\mathop{\rm Im}\nolimits}

\DeclareMathOperator{\Hilb}{Hilb}

\DeclareMathOperator{\Res}{Res}
\DeclareMathOperator{\Ext}{Ext}
\DeclareMathOperator{\Id}{Id}

\DeclareMathOperator{\PGL}{PGL}

\DeclareMathOperator{\Int}{Int}
\DeclareMathOperator{\ev}{ev}

\def\llra{\hbox to 10mm{\rightarrowfill}}
\def\lllra{\hbox to 15mm{\rightarrowfill}}

\newtheorem{lemm}{Lemma}[section]
\newtheorem{theo}[lemm]{Theorem}
\newtheorem{coro}[lemm]{Corollary}
\newtheorem{prop}[lemm]{Proposition}
\newtheorem*{claim}{Claim}

\theoremstyle{definition}

\newtheorem{rema}[lemm]{Remark}

\theoremstyle{remark}
\newtheorem*{remark*}{Remark}
\newtheorem*{note*}{Note}

 \def\moins{\mathop{\hbox{\vrule height 3pt depth -2pt
width 5pt}\,}}

\title{Hyper-K\"ahler fourfolds and Grassmann geometry}
\author{Olivier Debarre}
\address{\'Ecole Normale Sup\'{e}rieure\\
D\'{e}partement de Math\'{e}matiques et Applications\\
45 rue d'Ulm\\
75230 Paris cedex 05 -- France}
\email{odebarre@dma.ens.fr}

\author{Claire Voisin}
\address{IH\'{E}S and Institut de Math\'{e}matiques de Jussieu, 175 rue du Chevaleret, 75013 Paris, France}
\email{voisin@math.jussieu.fr}
\subjclass[2000]{14J35, 14M15, 14J70.}

\begin{document}

\maketitle

   \begin{abstract}
   We construct a new 20-dimensional family of algebraic hyper-K\"ahler fourfolds and prove that they are
  deformation-equivalent to the second punctual Hilbert scheme of a K3 surface of degree 22.

   \end{abstract}

\section{Introduction}
An irreducible hyper-K\"ahler manifold is a compact K\"ahler manifold whose space of holomorphic 2-forms is generated by
an everywhere  nondegenerate form. It is known, as a consequence of the Kodaira embedding theorem and the study of the period map,
that   algebraic hyper-K\"ahler manifolds form a countable union of hypersurfaces in the
local universal deformation space of any hyper-K\"ahler manifold.

In \cite{be}, Beauville described, in each dimension $2n$, two series of such varieties:
\begin{enumerate}
\item    the $n$-th punctual Hilbert scheme $S^{[n]}$ of a K3 surface $S$;
\item   the fiber at the origin of the Albanese map   of the $(n+1)$-st punctual Hilbert scheme of an abelian surface.
\end{enumerate}
All of the irreducible  hyper-K\"ahler manifolds constructed later on  have been  proved to be deformation-equivalent to one of Beauville's examples, with the exception of  two sporadic families of  examples constructed by O'Grady in \cite{ogrady2}, { in dimensions 6 and 10.}

Beauville's examples all have, in dimension at least 4,  Picard number $\geq 2$, while  a very general algebraic deformation
has Picard number $1$, hence is not of the same type. There are very few explicit geometric descriptions for these  deformations. More precisely, there are,
 to our knowledge, only three such families that are explicitly described, each of which is  20-dimensional and parametrizes general polarized deformations of the second punctual Hilbert scheme of a K3 surface:
\begin{enumerate}
\item Beauville and Donagi  proved in \cite{bedo} that the variety of lines $F(X)$ on a smooth cubic hypersurface  $X\subset \P^5$ is an algebraic hyper-K\"ahler fourfold.
This gives a 20-dimensional  moduli space of fourfolds, and along an explicitly described hypersurface  in this moduli space (corresponding to ``Pfaffian'' cubics), $F(X)$
  is isomorphic to the second punctual Hilbert scheme of a general K3 surface $S$ of degree 14.
\item Iliev and Ranestad proved in \cite{ilievra}
that the variety $V(X)$ of sum of powers of a general cubic  $X\subset \P^5$ as above is another
algebraic hyper-K\"ahler fourfold, with $20$ moduli. Along another hypersurface in the moduli space (corresponding to ``apolar'' cubics),
$V(X)$ is also isomorphic to $S^{[2]}$.
While the Hodge structure on  $H^2(V(X),\Z)$  is presumably isogenous to that of  $H^2(F(X),\Z)$ (a fact which is not known),
 it is shown in \cite{ilievra2}
 that the  polarization on $V(X)$ is in general numerically different from the Pl\"ucker polarization on $F(X)$. This guarantees that we have
two different families of deformations of $S^{[2]}$.
\item O'Grady  constructed in \cite{ogrady} a $20$-parameter family of hyper-K\"ahler algebraic fourfolds. They are
quasi-\'{e}tale double covers of certain sextic hypersurfaces  constructed by Eisenbud, Popescu, and Walter, and are deformations of the second punctual Hilbert scheme of a general K3 surface of degree 10.
\end{enumerate}

Our purpose in this paper is to construct and study  another  family of hyper-K\"ahler fourfolds, which is close in spirit to the
Beauville-Donagi family:  it is related to the geometry of Grassmannians, and there is an associated Fano hypersurface which
will play the role of the cubic hypersurface in \cite{bedo}.
The Grassmannian considered here is $G(6,V_{10})$, which parametrizes vector subspaces of dimension 6 of a fixed   vector space
$V_{10}$ of dimension 10.
Our starting point, which came to us following a discussion with
 Peskine,  is a $3$-form $\sigma \in \bigwedge^{3}V_{10}^*$. A dimension count shows that the moduli space of such $\sigma$
is $20$-dimensional.

We associate with  $\sigma$ two varieties: a hypersurface $F_\sigma$ in $G(3,V_{10})$, and a fourfold
$Y_\sigma$ in $G(6,V_{10})$.
Our first result is the following.

\begin{theo}
  There is a natural correspondence
$G_\sigma\subset Y_\sigma\times F_\sigma$, which is of relative dimension $9$ over $Y_\sigma$. When $Y_\sigma$ and $F_\sigma$ are smooth of the expected dimension,
 this correspondence induces an isomorphism of rational Hodge structures:
$$H^{20}(F_\sigma,\Q)_{\rm van}\isom  H^2(Y_\sigma,\Q)_{\rm van}.$$
The Hodge structure on the left-hand-side has Hodge numbers $h^{9,11}=h^{11,9}=1$ and  $h^{10,10}=20$, the other Hodge numbers being $0$.
\end{theo}

As a consequence, we conclude that $Y_\sigma$ is an irreducible  hyper-K\"ahler fourfold with second Betti number
$23$.
Although the construction of $Y_\sigma$ allows us to construct explicit  hypersurfaces
in the moduli space where its Picard number jumps to $2$ (see sections \ref{section1}, \ref{section2}, and
\ref{section5}), we have not been able to identify
an explicit hypersurface in the moduli space  where $Y_\sigma$ is isomorphic to the second punctual Hilbert scheme
of a K3 surface.
We prove however that $Y_\sigma$ is a deformation of such a  Hilbert scheme.

\begin{theo}\label{th1}
 The varieties $Y_\sigma$,  endowed  with the Pl\"ucker
 polarization, are deformation-equivalent to the second punctual Hilbert scheme $S^{[2]}$ of a K3 surface $S$ of degree 22, endowed with
 the polarization  whose pull-back to $\widetilde{S\times S}$ is $(\cO_S(1)\boxtimes \cO_S(1))^{10}(-33 \widetilde  E)$.
\end{theo}

Here, $\widetilde{S\times S}\to S\times S$ is the blow-up of the diagonal and $\widetilde E$ is the exceptional divisor; there is a canonical double cover $\widetilde{S\times S}\to S^{[2]}$.
The proof of this result  is closely related to that of the main result
of \cite{huy}, where Huybrechts  proved that birational equivalence implies deformation equivalence for
irreducible hyper-K\"ahler manifolds. However, we are in a situation where only a singular degeneration
of $Y_\sigma$ is birationally equivalent to  the second punctual Hilbert scheme of a K3 surface, to which we cannot apply directly Huybrechts' theorem.

To close this introduction, we would like to explain how our construction fits into the general results of  \cite{ghs} on the Kodaira dimension of certain modular varieties dominated by moduli spaces of hyper-K\"ahler manifolds (we would like to thank Hulek for pointing this out to us). We quickly review some of the relevant results of \cite{ghs}.

Let $S $ be a K3 surface  and let 
  $L$ be the  rank-23 lattice $H^2(S^{[2]},\Z)$, equipped with the Beauville-Bogomolov quadratic form $q$ (\cite{be}). Depending on the positive integer $d$, there are, under the action of the stable orthogonal group of $L$, either one or two orbits of primitive vectors $ h$ of $L$ with $q(h)=2d$: one is called {\em of split type} and the other {\em of nonsplit type} (it occurs if and only if $d\equiv -1 \pmod 4$).  Polarized hyper-K\"ahler manifolds  which are deformation equivalent to   $(S^{[2]},h)$ admit a quasi-projective coarse moduli space $\cM_h$ which is  finite over a dense open subset of a  locally symmetric modular variety $\cS_h$. When $h$ is of split type, it is proved in \cite{ghs} that $\cS_h$ (hence also $\cM_h$) is of general type for $d\ge 12$ and of nonnegative Kodaira dimension  for $d=9$ or $11$. 
  
  On the other hand, for the polarization $h$ mentioned in Theorem \ref{th1}, we have
  $$ q(h)  =100 \times 22 + (33)^2 \times(
-2)= 22,$$
hence $d=11$.  Furthermore, $h$ is of nonsplit type, and our construction proves that $\cM_h$ (hence also $\cS_h$) is unirational. 

\begin{rema} Part of the results of this paper (and particularly those concerning
the Hodge theory of the hypersurface $F_\sigma$) are  related to those of \cite{mmk}, where
hypersurfaces or complete intersections in homogeneous varieties with a Hodge structure
on middle cohomology of $3$-dimensional Calabi-Yau type
are exhibited and studied.
\end{rema}

\medskip

{\small\noindent{\bf Acknowledgements.}  This work was started at the MSRI during the algebraic geometry program of spring 2009. We thank
the organizers of this semester and the MSRI  for support and  the excellent environment provided during this period.
We also thank Christian Peskine for an inspiring discussion which gave us the starting point of this work, Daniel Grayson and Michael Stillman for their help with the program Macaulay2, Fr\'{e}d\'{e}ric Han for doing calculations for us with the program LiE, Laurent Manivel for his virtuosity with Bott's theorem, and Klaus Hulek for bringing to our attention the link between our construction and the results of  \cite{ghs}.}

\medskip

\noindent{\bf Notation.}  If $V$ is a complex vector space, we denote by $G(d,V)$ the Grassmannian of vector subspaces of $V$ of dimension $d$, by $\cS_d$   the rank-$d$ tautological vector subbundle on $G(d,V)$, and by $\cE_d$ its dual.

\section{The hypersurface $F_\sigma$ and the fourfold $Y_\sigma$}\label{section1}

Let $V_{10}$ be a (complex) vector space of dimension 10 and let $\sigma$ be a general element in $ \bigwedge^3V_{10}^*$.
The $3$-form $\sigma$ determines a Pl\" ucker hyperplane section
$$F_\sigma\subset G(3,V_{10})\subset \P( \bigwedge^3V_{10})$$
consisting of $3$-dimensional vector subspaces of $V_{10}$ on which $\sigma$ vanishes.

On the other hand, $\sigma$ determines a subvariety
$$Y_\sigma\subset G(6,V_{10})$$ defined as the set of
$6$-dimensional vector subspaces of $  V_{10}$ on which $\sigma$ vanishes identically.
It is the zero-set of a general section of
$\bigwedge^3\cE_6$.
As $\cE_6$ is generated by global sections,
$Y_\sigma$ is smooth and connected, of codimension $\rk(\bigwedge^3\cE_6)=20$.

We denote by $\cO_{G(6,V_{10})}(1)=\det(\cE_6)$ the Pl\" ucker line bundle on
$G(6,V_{10})$.
As $\omega_{G(6,V_{10})}=\cO_{G(6,V_{10})}(-10)$ and $\det( \bigwedge^3\cE_6)=\cO_{G(6,V_{10})}(10)$,
we conclude by adjunction that $Y_\sigma$ is a smooth fourfold with trivial canonical bundle.

Next we observe that there is a natural correspondence
between $F_\sigma$ and $Y_\sigma$. Namely,
each point of $Y_\sigma$ determines a $6$-dimensional vector subspace
$W_6\subset V_{10}$ on which
$\sigma$ vanishes identically, hence an inclusion
$G(3,W_6)\subset F_\sigma$.
Putting this together in a family gives us
a variety
$$G_\sigma=\{([W_3],[W_6])\in G(3,V_{10})\times G(6,V_{10}) \mid W_3\subset W_6,\ \sigma\vert_{W_6}=0\},$$
with two projections
\begin{equation}\label{Gsigma}
Y_\sigma\stackrel{q}{\longleftarrow}G_\sigma\stackrel{p}{\lra} F_\sigma.
 \end{equation}
The fibers of $p$ are the $9$-dimensional Grassmannians
$G(3,W_6)$. There is
thus
an induced cohomological correspondence
$$q_*p^*:H^{20}(F_\sigma,\Q)\to H^2(Y_\sigma,\Q),$$
whose restriction to vanishing cohomology will be denoted by
\begin{eqnarray}
\label{pq}
(q_*p^*)_{\rm van}:H^{20}(F_\sigma,\Q)_{\rm van}\to H^2(Y_\sigma,\Q),
\end{eqnarray}
where, if we denote by $j$ the inclusion $F_\sigma\hookrightarrow G(3,V_{10})$,
$$H^{20}(F_\sigma,\Q)_{\rm van}:=\Ker \bigl (H^{20}(F_\sigma,\Q)\stackrel{j_*}{\to}H^{22}(G(3,V_{10}),\Q)\bigr).$$

Our aim in this section is to investigate the geometry of $Y_\sigma$ and of the correspondence
introduced above.
We will show the following.

\begin{theo}\label{depart} The variety $Y_\sigma$ is an irreducible hyper-K\"ahler fourfold with $b_2=23$.
\end{theo}

This means by  definition that
$Y_\sigma$ has an everywhere  nondegenerate holomorphic $2$-form, unique up to multiplication by a nonzero scalar,
and, because $Y_\sigma$ has trivial canonical bundle, this is equivalent by \cite{be} and \cite{bo}
to
$h^{2,0}(Y_\sigma)\ne 0$ and no finite cover of $Y_\sigma$ is a product of
two algebraic K3 surfaces, or the product of an abelian surface by an
 algebraic K3 surface, or an abelian fourfold.

The first  step in the proof is the following result concerning the geometry
 of $F_\sigma$.

\begin{theo} \label{cohFsigma} 1) The only nonzero Hodge numbers of the Hodge structure on $H^{20}(F_\sigma,\Q)_{\rm van}$ are
$$h^{9,11}(F_\sigma)=h^{11,9}(F_\sigma)=1\quad{\it and}\quad h^{10,10}(F_\sigma)_{\rm van}=20.$$

2) For $\sigma$ very general,
the Hodge structure on $H^{20}(F_\sigma,\Q)_{\rm van}$ is simple.

 3) The morphism of Hodge structures $(q_*p^*)_{\rm van}$ in (\ref{pq}) is injective.
\end{theo}

\begin{proof}
1)  This is an immediate consequence of Griffiths' description of the Hodge structure on the
vanishing cohomology of an ample hypersurface (see \cite{grirat} and \cite{voisinbook}, 6.1.2).
Let $U:=G(3,V_{10})\moins F_\sigma$. We have first of all the following.

\begin{lemm}
 The restriction map
$$H^{20}(G(3,V_{10}),\Q)\to H^{20}(U,\Q)$$
is zero.
\end{lemm}

\begin{proof}
The cohomology of $G(3,V_{10})$ is generated as an algebra by the classes
$\ell=c_1(\cS_3)$, $c_2=c_2(\cS_3)$, and $  c_3=c_3(\cS_3)$,
where $\ell$ is proportional to the class of $F_\sigma$, hence vanishes on $U$.
On the other hand, consider the
projective bundle $\P(\cS_3)$ on $G(3,V_{10})$. It admits a natural
map $\alpha$ to $\P(V_{10})$ and its cohomology is generated
by $h=\alpha^*c_1(\cO_{\P(V_{10})}(1))$ as an algebra over
$H^\bullet(G(3,V_{10}),\Q)$, with the sole relation
$$h^3+h^2\ell +hc_2+c_3=0.$$
Modulo $\ell$, hence  in  $H^\bullet(U)$, this relation becomes
\begin{eqnarray*}
h^3+hc_2+c_3=0.
\end{eqnarray*}
Together with the vanishing $h^{10}=0$, this yields
the following
equalities in $H^\bullet(U,\Q)$:
$$c_2^4=3c_2c_3^3,\  c_3^3=4c_3c_2^3,\ c_2^2c_3^2=0.$$
But the only polynomials of weighted degree $10$ in $c_2$ and $c_3$
are
$c_2^5$ and $c_2^2c_3^2$, and they vanish by the relations above.
\end{proof}

This lemma and the Thom exact sequence (\cite{voisinbook}, 6.1.1)  show that the residue
map is an isomorphism
$$H^{21}(U,\Q)\isom H^{20}(F_\sigma,\Q)_{\rm van}.$$
Now we apply Griffiths' theory (\cite{grirat}; see also \cite{voisinbook}, 6.1.2), which describes the Hodge filtration
on the cohomology $H^{21}(U,\Q)$ (which up to a shift of $-1$
corresponds to the
Hodge filtration on $H^{20}(F_\sigma,\Q)_{\rm van}$).

The only assumption we need is the vanishing
$$ H^i(G(3,V_{10}),\Omega_{G(3,V_{10})}^j(k))=0\quad{\rm for\ all\ } k>0,\ i>0,\ j\ge 0,$$
which we get from Bott's theorem. 
It   follows that
$$F^pH^{20}(F_\sigma,\C)_{\rm van}=F^{p+1}H^{21}(U,\C)$$
is generated by residues
$$\Res_{F_\sigma}\frac{\alpha}{\sigma^{21-p}},$$
where $\alpha$ runs through the space of  sections of $\omega_{G(3,V_{10})}(21-p)=\cO_{G(3,V_{10})}(11-p)$.
We immediately get the vanishing of $F^{12}H^{20}(F_\sigma,\C)_{\rm van}$, hence of
$h^{p,20-p}(F_\sigma,\C)_{\rm van}$ for $p\ge 12$.

For $p=11$, we get a $1$-dimensional vector space
generated by $\Res_{F_\sigma}\frac{\alpha}{\sigma^{10}}$,
where $\alpha$ is a nowhere vanishing section of $\omega_{G(3,V_{10})}(10)=\cO_{G(3,V_{10})}$.
For $p=10$, we find that
$H^{10,10}(F_\sigma,\C)_{\rm van}$ is generated by the residues
$$\Res_{F_\sigma}\frac{\alpha}{\sigma^{11}},$$
where $\alpha$ runs through the space of  sections  of $\omega_{G(3,V_{10})}(11)=\cO_{G(3,V_{10})}(1)$.
Finally, we recall the analysis (adapted from \cite{grirat}; see also \cite{voisinbook},
6.1.3, where the case of hypersurfaces in a projective space is treated) of the
kernel of the maps
\begin{eqnarray*}
H^0(G(3,V_{10}),\cO_{G(3,V_{10})})\isom
H^0(F_\sigma,\cO_{F_\sigma})&\to&  H^{11,9}(F_\sigma)\\
\alpha&\mapsto& \Res_{F_\sigma}\frac{\alpha}{\sigma^{10}}
\end{eqnarray*}
and
\begin{eqnarray*}
 H^0(G(3,V_{10}),\cO_{G(3,V_{10})}(1))/\C\sigma\isom
H^0(F_\sigma,\cO_{F_\sigma}(1))&\to& H^{10,10}(F_\sigma)\\
 \alpha\ \ \mapsto\ \ \Res_{F_\sigma}\frac{\alpha}{\sigma^{11}}  \pmod{H^{11,9}(F_\sigma)}
\end{eqnarray*}
 induced by the residue maps.
The same analysis as in the case of hypersurfaces in a projective space shows that the kernels
are Jacobian ideals obtained respectively from sections of $T_{G(3,V_{10})}(-1)$ and sections of $T_{G(3,V_{10})}$
via the natural maps
$$H^0(G(3,V_{10}),T_{G(3,V_{10})}(l))\to H^0(F_\sigma,\cO_{F_\sigma}(l+1)),$$
for $l\in\{-1,0\}$, induced by the normal bundle exact sequence of $F_\sigma$.

Now $H^0(G(3,V_{10}),T_{G(3,V_{10})}(-1))=0$, whereas the vector space\break $H^0(G(3,V_{10}),T_{G(3,V_{10})})$ has dimension
$99$ and injects into $H^0(F_\sigma,\cO_{F_\sigma}(1))$.
Hence we conclude
$h^{10,10}(F_\sigma)_{\rm van}=119-99=20$.

\medskip

2) The simplicity of a polarized Hodge structure of weight $20$ with Hodge numbers
$h^{11,9}=1$, and $h^{i,20-i}=0$ for $i>11$, is equivalent to the fact that there
are no Hodge classes in $H^{10,10}$ (here we use the polarization to say that  any nontrivial
 Hodge substructure has $h^{11,9}=0$, hence consists of Hodge classes, or
its orthogonal complement has $h^{11,9}=0$, hence consists of Hodge classes).
So it suffices to prove that for $\sigma$ very general, there are no Hodge classes
in $H^{20}(F_\sigma,\Q)_{\rm van}$. This is a Noether-Lefschetz type theorem which is proved by
the classical Lefschetz  monodromy argument (see \cite{voisinbook}, 3.2.3).

\medskip

3) By simplicity, the morphism of Hodge structures $(q_*p^*)_{\rm van}$ is either $0$ or injective.
It thus suffices to prove that it is not   $0$.
Equivalently, it  suffices to prove  that the morphism
$$q_*p^*:H_2(Y_\sigma,\Q)\to H_{20}(F_\sigma,\Q)$$
has rank at least $2$. Indeed, since $H_2(G(6,V_{10}),\Q)$ has dimension $1$,
denoting by $i_\sigma:Y_\sigma\to G(6,V_{10})$ the inclusion,
we find that
$q_*p^* $ has rank at least $ 2$ if and only if its restriction
$$q_*p^*\vert_{ \Ker i_{\sigma*}}:
\Ker i_{\sigma*}\to H_{20}(F_\sigma,\Q)$$ has rank at least
$1$. But this   morphism takes its values in $H_{20}(F_\sigma,\Q)_{\rm van}$ and its
dual is the morphism $(q_*p^*)_{\rm van}$   composed with the inclusion of $(\Ker i_{\sigma*})^*$ into $H^2(Y_\sigma,\Q) $.

We make now the following construction. Consider
subspaces $V_4\subset V_7\subset V_{10}$, where the subscripts
indicate the dimension, and choose  $\sigma\in \bigwedge^3V_{10}^*$ satisfying
$$\sigma\vert_{V_7}=\alpha_1\wedge\alpha_2\wedge\alpha_3\quad {\rm and}\quad V_4=\{\alpha_1=\alpha_2=\alpha_3=0\}.$$
One verifies that one can choose such a $\sigma$ keeping $Y_\sigma$ and $F_\sigma$ smooth.

In this situation, $Y_\sigma$ contains a line (with respect to the Pl\"ucker embedding);
namely, choosing any $V_5$ such that
$V_4\subset V_5\subset V_7$, and observing that $\sigma $ vanishes on any hyperplane of
$V_7$ containing $V_4$, we find that the line
$$C=\{ [W_6]\mid V_5\subset W_6\subset V_7\} $$
is contained in $Y_\sigma$.

Let $Z=q(p^{-1}(C))$. Observe that the class $z\in H_{20}(F_\sigma,\Q)$ of $Z$ is
equal to $q_*p^*c$, where $c$ is the class of $C$. Furthermore, the classes so obtained
 are in the same orbit under the monodromy action.

We will now specialize $\sigma$ further in two ways, asking that $Y_\sigma$ contain
two curves $C$ and $C'$ as above (but of course     { with different cohomology classes in $Y_\sigma$}).

\medskip

A) We choose $V_4\subset V_7\subset V_{10}$ and $V'_4\subset V'_7\subset V_{10}$ in such a way
that the intersection $V_7\cap V'_7$ is transverse, and   $V_4\cap V'_4=\{0\}$.
In a suitable basis $(e_1,\ldots, e_{10})$ of $V_{10}$,
we take
$$V_7=\langle e_1,\ldots,e_7\rangle,\ V_4=\langle e_2,\ldots, e_5\rangle,\ V'_7=\langle e_4,\ldots, e_{10}\rangle,\ V'_4=\langle e_6,\ldots, e_9\rangle.$$
Then,
$$\sigma\vert_{V_7}=e_1^*\wedge e_6^*\wedge e_7^*\quad {\rm and}\quad \sigma\vert_{V'_7}= e_4^*\wedge e_5^*\wedge e_{10}^*,$$
and this is compatible, because on the intersection
$$V_7\cap V'_7=\langle e_4,\ldots, e_7\rangle,$$
the two 3-forms $e_1^*\wedge e_6^*\wedge e_7^*$ and $e_4^*\wedge e_5^*\wedge e_{10}^*$ vanish.
One verifies that for a general choice of $\sigma$ as above, $Y_\sigma$ and $F_\sigma$ are smooth.

\medskip

B) We choose $V_4\subset V_7\subset V_{10}$ and $V'_4\subset V'_7\subset V_{10}$ in such a way
that the intersection $V_7\cap V'_7$ is transverse, but   $V_4\cap V'_4$
is $1$-dimensional.
In a suitable basis $(e_1,\ldots, e_{10})$ of $V_{10}$,
we take
\begin{eqnarray*}
&V_7=\langle e_1,\ldots,e_7\rangle,\ V_4=\langle e_1,\ldots, e_4\rangle,&\\
& V'_7=\langle e_4,\ldots,e_{10}\rangle,\ V'_4=\langle e_4,e_8, e_9,e_{10}\rangle.&
\end{eqnarray*}
Then,
$$\sigma\vert_{V_7}=e_5^*\wedge e_6^*\wedge e_7^*\quad {\rm and}\quad \sigma\vert_{V'_7}= e_5^*\wedge e_6^*\wedge e_7^*$$
are obviously compatible and we indeed have a $1$-dimensional intersection $V_4\cap V'_4$, generated by $e_4$.

One checks that for a general choice of $\sigma$ as above, $Y_\sigma$ and $F_\sigma$ are smooth.

The proof of the theorem is then concluded by the following lemma.

\begin{lemm} \label{petitlemme} The classes $z$, $z'\in H_{20}(F_\sigma,\Q)$ constructed above
satisfy
$$ z\cdot z'=0$$
in situation A), and
$$z\cdot z'=1$$
in situation B).
\end{lemm}
Indeed, if $q_*p^*$ has rank $1$, the classes $z$ and $z'$ must be proportional. As they are
in the same orbit of the monodromy action, they are equal. This contradicts the fact
that they satisfy $z\cdot z'=0$ or $z\cdot z'=1$ according to the configuration.
\end{proof}

\begin{proof}[Proof of Lemma \ref{petitlemme}.] Note that in both cases, the (singular) variety $Z$ is described
as follows:
$$Z=\{W_3\subset V_7\mid \dim (W_3\cap V_5) \ge 2\}.$$
In situation A), we may choose $V_5$ and $V'_5$   transverse, so that $V_5\cap V'_5=\{0\}$. But then,
$$Z\cap Z'=\{W_3\subset V_7\cap V'_7\mid \dim (W_3\cap V_5 )\ge 2,\ \dim ( W_3\cap V'_5) \ge 2\}$$
is clearly empty.

In situation B), we may choose $V_5$ and $V'_5$ so that they meet along the $1$-dimensional vector space
$\langle e_4\rangle=V_4\cap V'_4$. Then
$$Z\cap Z'=\{W_3\subset V_7\cap V'_7\mid \dim (W_3\cap V_5) \ge 2,\ \dim ( W_3\cap V'_5) \ge 2\},$$
and denoting by $V_{5,0}$ (resp. $V'_{5,0}$) the $2$-dimensional intersection$V_5\cap V'_7$ (resp. $V'_5\cap V_7$),
we find
$$Z\cap Z'=\{W_3\subset V_7\cap V'_7\mid \dim ( W_3\cap V_{5,0} )\ge 2,\ \dim ( W_3\cap V'_{5,0} )\ge 2\}.$$
As $V_{5,0}$ and $V'_{5,0}$ are $2$-dimensional, one must have for such a $W_3$:
$$V_{5,0}=W_3\cap V_{5,0}\quad {\rm and}\quad  V'_{5,0}=W_3\cap V'_{5,0},$$
and finally, $W_3=V_{5,0}+V'_{5,0}$.

Thus the intersection $Z\cap Z'$ consists of one point, namely
the point $[V_{5,0}+V'_{5,0}]$ of $G(3,V_{10})$, and it follows
that $z\cdot z'$ is nonzero in this case.  To prove    $z\cdot z'=1$, one notes that $Z$ and $Z'$ are smooth
at the above point, and one checks that the intersection is transverse.
\end{proof}

\begin{rema}    The hyper-K\"ahler manifolds $Y_\sigma$ containing a line as above are very similar
to the Fano varieties of lines in a cubic fourfold (\cite{bedo}) containing a plane  (\cite{vlag}). Indeed, the $V_5$ introduced in the construction
of the line $C$ varies in the plane  $\P(V_7/V_4)$. Furthermore, the subset of
$Y_\sigma$ swept out by the curves $C$   is the  dual plane $\P((V_7/V_4)^*) \subset Y_\sigma$ parametrizing
hyperplanes of $V_7$ containing $V_4$. This, as noticed in \cite{vlag}, is a Lagrangian plane in $Y_\sigma$.
\end{rema}

\begin{proof}[Proof of Theorem \ref{depart}] Theorem \ref{cohFsigma} implies  $h^{2,0}(Y_\sigma)\ne 0$.
In order to show that $Y_\sigma$ is an  irreducible hyper-K\"ahler variety, it thus suffices to show that no finite \'{e}tale cover of $Y_\sigma$
is an abelian fourfold, the product of an abelian surface and an algebraic
K3 surface,
or the product of two algebraic K3 surfaces. But this follows again from Theorem \ref{cohFsigma}. Indeed,
this theorem implies that the Hodge structure on $H^2(Y_\sigma,\Q)$ contains an irreducible Hodge substructure with $h^{1,1}=20$. If such a covering existed,   this irreducible Hodge structure
would inject into the transcendental part of
the $H^2$ of  an abelian fourfold, an abelian
surface, or an algebraic K3 surface, where ``transcendental'' means ``orthogonal to
the set of Hodge classes in the  Poincar\'{e} dual cohomology group.''
But the Hodge structures on the transcendental part of the $H^2$ of an abelian fourfold, an abelian
surface, or an algebraic K3 surface all have $h^{1,1}_{\rm tr}\le 19$.

To conclude the proof of the theorem, we need to show $b_2(Y_\sigma)=23$.
We already know   $b_2(Y_\sigma)\ge 23$: indeed, the image of $(q_*p^*)_{\rm van}$ has rank
$22$ and it is not the whole of $H^2(Y_\sigma,\Q)$ because it does not contain any Hodge class
for a very general $\sigma$.  As  $Y_\sigma$ is an irreducible hyper-K\"ahler fourfold, the equality $b_2(Y_\sigma)=23$ then
 follows from \cite{guan}, where it is  proved that $23$ is the maximal possible second Betti number.
\end{proof}

\begin{rema}\label{manhan}   It is also possible (and even shorter) to prove that $Y_\sigma$ is a hyper-K\"ahler variety by showing
$$\chi(Y_\sigma, \cO_{Y_\sigma})=\sum_{i=0}^{20}(-1)^i\chi\Big(G(6,V_{10}), \bigwedge^i\bigl(\bigwedge^3\cE_6\bigr)\Big)=3,$$   using for example
Macaulay. Alternatively,  as shown to us by Manivel and Han,  using the Koszul resolution of $\cO_{Y_\sigma}$, Bott's theorem, and properties of the irreducible representations that occur in  $\bigwedge^i(\bigwedge^3V_6)$ (or, alternatively, the program LiE), one can prove directly $h^2(Y_\sigma, \cO_{Y_\sigma})=1$. However, the  proof above is more geometric
and  explains where the holomorphic $2$-form comes from.
\end{rema}

To conclude this section, note that Theorem \ref{depart} allows us in turn to refine Theorem \ref{cohFsigma} as follows. Consider again the inclusion
$i_\sigma$ of $Y_\sigma$ into $G(6,V_{10})$ and   define the vanishing
cohomology $H^2(Y_\sigma,\Q)_{\rm van}$ as the kernel of
$$ i_{\sigma*}:H^2(Y_\sigma,\Q)\to
H^{42}(G(6,V_{10}),\Q)\isom H_6(G(6,V_{10}),\Q).$$

\begin{coro}  \label{cor6avril} The morphism $i_{\sigma*}$ has rank $1$.
 The morphism
of Hodge structures
$(q_*p^*)_{\rm van}$ defined in (\ref{pq}) takes values in $H^2(Y_\sigma,\Q)_{\rm van}$ and induces an isomorphism
$$H^{20}(F_\sigma,\Q)_{\rm van}\isom H^2(Y_\sigma,\Q)_{\rm van}.$$
\end{coro}

\begin{proof}
The composition
$$i_{\sigma*}\circ (q_*p^*)_{\rm van}:H^{20}(F_\sigma,\Q)_{\rm van}\to H^{42}(G(6,V_{10}),\Q)$$
vanishes, because the Hodge structure on the right-hand-side is trivial, while the Hodge structure
on the left-hand-side is nontrivial and generically simple. Thus $(q_*p^*)_{\rm van}$ takes values
in $H^2(Y_\sigma,\Q)_{\rm van}$.
We know that this morphism is injective and that the left-hand-side has dimension
$22$. Hence all the statements follow from the equality $b_2(Y_\sigma)=23$, so that
$\dim(\Ker i_{\sigma*})\le 22$, with equality if and only if $\rk(i_{\sigma*})=1$ and
$(q_*p^*)_{\rm van}$ surjects onto $H^2(Y_\sigma,\Q)_{\rm van}$.
\end{proof}

\section{Singular hypersurfaces $F_\sigma$\label{section2}}
In this section, we are interested in those $\sigma\in \bigwedge^3V_{10}^*$ for which the hypersurface  $F_\sigma\subset G(3,V_{10})$ is  {\em singular.}

\begin{prop}\label{sing}
The dual variety $G(3,V_{10})^*\subset \P(\bigwedge^3V_{10}^*)$ is an irreducible hypersurface. For $\sigma$ general in  $G(3,V_{10})^*$, the corresponding hyperplane section $F_\sigma$ of $G(3,V_{10})$ has a unique singular point. It corresponds
to a $3$-dimensional vector subspace $W\subset V_{10}$ such that $\sigma\vert_{\bigwedge^2W  \wedge V_{10}}=0$.
\end{prop}

\begin{proof}The fact that the dual variety $G(3,V_{10})^*\subset \P(\bigwedge^3V_{10}^*)$ is a hypersurface follows for example from \cite{las}, \S3, which proves that its degree is 640. Then it is classical that this hypersurface is irreducible, and that a general  point corresponds to a hyperplane tangent to  $G(3,V_{10})$ at a single point.

Let $[W]$ be a point of $G(3,V_{10})$. The embedding  of
$$T_{G(3,V_{10}),[W]}\isom \Hom(W,V_{10}/W)$$
 into
 $$T_{\P(\bigwedge^3V_{10}),[\bigwedge^3W]}\isom\Hom(\bigwedge^3W,\bigwedge^3V_{10}\Big/\bigwedge^3W)$$
  is given by
$$ u\mapsto \bigl(  w_1\wedge w_2\wedge w_3 \mapsto  u(w_1)\wedge w_2\wedge w_3+w_1\wedge u(w_2)\wedge w_3+ w_1\wedge w_2\wedge u(w_3)\bigr).
$$
Therefore, the hyperplane section $F_\sigma\subset G(3,V_{10}) $ defined by $\sigma\in \bigwedge^3V_{10}^*$ is singular at   $[W]$ if and only if  $\sigma(w_1\wedge w_2\wedge v)=0$ for all $w_1$, $w_2$ in $W$ and all $v\in V_{10}$.
\end{proof}

We will henceforth assume that $\sigma$ corresponds to a general point of the
discriminant  hypersurface $G(3,V_{10})^*$ and we denote by $[W]$ the unique singular point of
 $F_\sigma$. By Proposition \ref{sing}, we have
 $$\sigma\vert_{\bigwedge^2W \wedge V_{10}}= 0.$$

  For each $d\in\{0,1,2,3\}$, let  $Y^d_\sigma$ be the union of those components $Y'_\sigma$ of $Y_\sigma$ such that $\dim(W\cap W_6)=d$ for $[W_6]$ general in $Y'_\sigma$.

   \begin{prop}\label{Yi}
1) The variety $Y^3_\sigma$ is a general K3 surface of   genus 12.

2) The varieties $Y^1_\sigma$ and $Y^2_\sigma$  are either empty or smooth of dimension $2$.

3) The variety $Y^0_\sigma$ is a smooth  and irreducible fourfold.

4) The variety $Y_\sigma$  is  a normal and irreducible fourfold.
 \end{prop}

 \begin{proof}
1) Choose a decomposition $V_{10}=W\oplus Q$. Since $ \sigma$ vanishes on $\bigwedge^2W \wedge V_{10}$, we can write $\sigma=\sigma_1+\sigma_2$
 with $\sigma_1\in W^*\otimes \bigwedge^2Q^*$ and $\sigma_2\in \bigwedge^3Q^*$. The projection $V_{10}\to Q$ induces an isomorphism between $Y^3_\sigma$ and
 \begin{equation}\label{sig}
 S=\{[W']\in G(3,Q)\mid  [W\oplus W']\in Y_\sigma\}.
  \end{equation}
This variety is defined by the vanishing
 of $\sigma_2$, viewed as a section of $\cO_{G(3,Q)}(1)$, and of $\sigma_1$, viewed as   a $3$-dimensional space of sections of $\bigwedge^2\cS_3^*$. Since $\sigma_1$ and $\sigma_2$ are general, $S$ is
   a general K3 surface  of genus 12  (\cite{muk}, Theorem 10). 



\medskip

2) This follows from a parameter count, which we will only do for $Y^2_\sigma$, the case of $Y^1_\sigma$ being completely analogous.  The dimension of the set of $([\sigma], [W],[W_1],[W_2],[W_4],[W'_3])$ such that $W=W_1\oplus W_2$, $V_{10}=W\oplus W_4\oplus W'_3$, and $\sigma\vert_{\bigwedge^2W \wedge V_{10}}$ and $\sigma\vert_{\bigwedge^3(W_2\oplus W_4)}$ vanish, \ie, with the notation above,
 \begin{eqnarray*}
 \sigma_1&\in& \Big(W_2\otimes \bigl((W_4\otimes W'_3)\oplus \bigwedge^2 W'_3\big)\Big)^* \oplus \Big(W_1\otimes \bigwedge^2 (W_4\oplus W'_3)\Big)^*
\\
\sigma_2&\in& \Big(\bigwedge^2W_4\otimes W'_3\Big)^*\oplus\Big(W_4\otimes \bigwedge^2W'_3\Big)^*\oplus\Big(  \bigwedge^3W'_3\Big)^*,
\end{eqnarray*}
is $30+21+18+12+1-1=81$  for the choice of $[\sigma]$, plus $9+16+24+21=70$ for the choices of $W_1$, $W_2$, $W_4$, and $W'_3$, hence 151. The set of $([\sigma], [W],[W_6])$ such that $F_\sigma$ is singular at $[W]$ and $[W_6]\in Y^2_\sigma$, is therefore smooth of dimension $151$ minus $2+8+21$ for the choices of $W_1$, $W_2$, and $W'_3$, hence $120$. For $[\sigma]$ general in the 118-dimensional hypersurface $G(3,V_{10})^*$, it follows by generic smoothness that $Y^2_\sigma$ is either empty, or smooth of dimension 2.

\medskip

3) Similarly, we consider  the set of $([\sigma], [W],[W_6],[W_1])$ such that $V_{10}=W\oplus W_6\oplus W_1$, and
  both  $\sigma\vert_{\bigwedge^2W  \wedge V_{10}}$ and $\sigma\vert_{\bigwedge^3W_6}$ vanish. It  is smooth, hence so is  the (122-dimensional) set of $([\sigma], [W],[W_6])$ such that $F_\sigma$ is singular at $[W]$, and $[W_6]\in Y^0_\sigma$. By generic smoothness, so is the general, 4-dimensional fiber $Y^0_\sigma$ of the projection $([\sigma], [W],[W_6])\mapsto ([\sigma], [W])$.

\medskip

4) Since $Y_\sigma$ has everywhere dimension at least $4$, the variety $Y^0_\sigma$ is dense in  $Y_\sigma$, which has therefore dimension $4$. It is moreover
 a local complete intersection, hence is connected in codimension $1$ (\cite{har}). It is also connected and, its singular locus being contained in the surface $Y^1_\sigma\sqcup Y^2_\sigma\sqcup Y^3_\sigma$, it is irreducible and normal.
 \end{proof}

Let  $p:V_{10}\to V_{10}/W$ be the canonical projection. The K3 surface $S$ of (\ref{sig}) is defined more canonically as
 \begin{equation}\label{sigma}
   S=\{[W']\in G(3,V_{10}/W)\mid  [p^{-1}(W')]\in Y_\sigma\}.
    \end{equation}
We now prove the main result of this section.

 \begin{theo}\label{phi}
 There is a birational isomorphism
 $$ \phi:S^{[2]}\dra  Y_\sigma$$
  defined as follows:
let $[W']$ and $[W'']$ be general points of $S$; then  $\phi([W'],[W''])$ is the only element $[W_6]$ of $ Y^0_\sigma$ such that $p(W_6)=  W'\oplus W''$.
    \end{theo}

 \begin{proof} We first show that the map $\phi^{-1}$ is well-defined at a general point $[W_6]$ of   $Y^0_\sigma$. We will show that there are exactly two points $[W']$ of $S$ such that $W'\subset p(W_6)$.

 Choose as above a decomposition $V_{10}=W\oplus Q$ with $W_6\subset Q$ and identify $V_{10}/W$ with $Q$. Let $W'$ be a 3-dimensional vector subspace of $W_6$.
  Since $\sigma$ vanishes on $\bigwedge^2W  \wedge V_{10} $ and on $\bigwedge^3W_6$, the condition $[W\oplus W']\in Y_\sigma$ is equivalent to the vanishing of $\sigma$ on $W\otimes \bigwedge^2W' $. This means that $[W_1']\in G(3,  W_6)$ is in the
zero locus of 3 sections of $\bigwedge^2\cS_3^*$. Since $c_3(\bigwedge^2\cS_3^*)^3=2$, it is either two or infinitely many points. As in the proof of Proposition \ref{Yi} above, one sees that it is in fact two points $[W']$ and $[W'']$ and that moreover, $W'+W''$ has dimension $6$, hence is equal to $W_6$. In other words, $\phi^{-1}$ is well-defined at the point $[W_6]$, which it maps to the unordered pair $([W'],[W''])$.

Conversely, let $[W']$ and $[W'']$ be general points of $S$. Choose again a splitting
 $V_{10}{ \isom} W\oplus V_{10}/W$.  Write a 6-dimensional vector subspace $W_6$ of $ W\oplus W'\oplus W''$ such that $W\cap W_6=\{0\} $  as the graph
$$ \{u(w',w'')+w'+w''\mid w'\in W',\ w''\in W''\}$$
 of some linear map $u:W'\oplus W''\to W$. The condition that $\sigma$ vanish on $W_6$ is then equivalent to the vanishing of the form $(\Id_{W'\oplus W''},u)^*\sigma\in \bigwedge^3(W'\oplus W'')^*$.
  Since $\sigma$ vanishes  on $\bigwedge^3(W\oplus W') $ and on $\bigwedge^3 (W\oplus W'') $,  this form is actually in
 $\left(\bigwedge^2 W' \otimes W''\right)^*\oplus \left( W' \otimes \bigwedge^2W''\right)^*$
 and depends in an affine way on $u$. In other words,  $[W_6]\in Y_\sigma$ if and only if $u$ is in the inverse image of $0$ by an {\em affine} map
 \begin{equation}\label{www}
 \Hom(W'\oplus W'',W)\stackrel{f_{W',W''}}{\lllra}\bigl( \bigwedge^2W'\otimes W''\bigr)^*\oplus \bigl(W'\otimes \bigwedge^2W''\bigr)^*.
 \end{equation}
Therefore, the set of elements $[W_6]$ of $ Y^0_\sigma$ such that $W_6\subset   W\oplus W'\oplus W''$ are (possibly empty) affine spaces.

The graph of $\phi^{-1}$ has dimension $4$ and dominates $ S^{[2]}$, and we just proved that the   fibers are affine spaces.
It follows that this projection is birational, hence $\phi^{-1}$ (and $\phi$) are birational isomorphisms.
 \end{proof}

We end this section with the computation of the line bundle
$\phi^*\cO_{Y_\sigma}(1)$ on $S^{[2]}$.   Recall that, if
$$\eps:\widetilde{S\times S}\to S\times S$$
is the blow-up of the diagonal,  $S^{[2]}$ can be seen as the quotient of
$\widetilde{S\times S}$ by the involution exchanging the two factors. We denote by
$$r :\widetilde{S\times S} \to S^{[2]} $$
 the quotient map, by   $\widetilde E\subset \widetilde{S\times S}$ the exceptional divisor of $\eps$, and by $E$ its image in $S^{[2]}$.

For $i\in\{1,2\}$, let  $p_i:\widetilde{S\times S}\to S$ be the $i$-th projection.
Given coherent sheaves  $\cF$ and $\cG$ on $S$, we define coherent sheaves on $\widetilde{S\times S} $ by setting
$$\cF\boxtimes\cG:=p_1^*\cF\otimes p_2^*\cG\quad{\rm and}\quad \cF\boxplus\cG:=p_1^*\cF\oplus p_2^*\cG.
$$

\begin{prop} \label{propphi*H}
The pull-back to $\widetilde{S\times S}$ of $\phi^*\cO_{Y_\sigma}(1)$ is isomorphic to
$(\cO_S(1)\boxtimes \cO_S(1))^{10}(-33 \widetilde  E)$.
\end{prop}

\begin{proof}
There are two natural  vector bundles of rank 6 on the open subset $U_1$ of $S^{[2]}$ where $\phi$   defines a morphism $\phi_{U_1}: U_1\to Y_\sigma\subset G(6,V_{10})$:
the pull-back $\phi_{U_1}^*(\cE_6\vert_{Y_\sigma})$ and $\cF_6\vert_{U_1}$, where
\begin{equation*}\label{es}
\cF_6:=r_*p_1^*(\cE_3\vert_S).
\end{equation*}
 Recalling from Theorem \ref{phi} the definition of $\phi$, observe that there is a natural morphism
$$ \label{Pdef}
P:\cF_6\vert_{U_1}\to \phi_{U_1}^*(\cE_6\vert_{Y_\sigma}).
$$
  induced by the dual of the projection   $p:V_{10}\to V_{10}/W$.
This implies
\begin{equation}\label{phi*}
\phi^*(\cO_{Y_\sigma}(1))=\det(\cE_6\vert_{Y_\sigma})=(\det\cF_6)(D),
\end{equation}
where $D$ is the divisor defined by the vanishing of the determinant of $P$.
Next, as the pull-back   of $\cF_6$ to $\widetilde{S\times S}$ fits into the exact
sequence
$$
0\to r^*\cF_6\to \cE_3\vert_S\boxplus \cE_3\vert_S\to \eps_{\widetilde  E}^*\cE_3\vert_S\to 0,
$$
where $\eps_{\widetilde E}:\widetilde E\to S$ is induced by the blow-up map $\eps$,
we get
\begin{equation}\label{detf}
\det(r^*\cF_6)=(\cO_S(1)\boxtimes \cO_S(1))(-3\widetilde E).
\end{equation}
It remains to analyze $D$.
We first compute the class of the divisor $D'$ where  the morphism   (\ref{www}), suitably defined, is not of maximal rank. It has the same support as $D$, and we will next compute their respective multiplicities.

\begin{lemm} \label{1lemme} The pull-back to $\widetilde{S\times S}$ of the divisor $D'$ is in the
linear system $\vert (\cO_S(1)\boxtimes \cO_S(1))^6(-20 \widetilde E)\vert $.
\end{lemm}

\begin{proof} In order to compute the full  class of $D'$ as a determinant, we need first  to extend the definition of $f_{W',W''}$ at a general point of $E$.

The rational map
$$S^{[2]}\dra G(6,V_{10}/W)$$
defined by the global sections of $\cF_6$ is well-defined on an open subset $U_2$ of $S^{[2]}$ whose complement has codimension $\ge 2$.
At a point $z\in U_2$, we may consider the fiber  $\cF_{6,z}$ as a hyperplane in $V_{10}/W$ which, when $z$ is a general point $([W'],[W''])$, is just $W'\oplus W''$.

On the other hand, there is a natural restriction map
$$R: \bigwedge^{3}\cF_6\to \cL_2,$$
where
$\cL_2$ is the rank-2 vector bundle $  r_*p_1^*(\bigwedge^3\cE_3)$ on $S^{[2]}$. The fiber of $\cL_2$ at a pair $([W'],[W''])$ away from $E$ is
the direct sum $(\bigwedge^3W'\oplus \bigwedge^3W'')^*$.

  At the point $z$, the map $u\mapsto (\Id_{W'\oplus W''},u)^*\sigma$ considered in the   proof of Theorem \ref{phi} is now a map
 $
 \Hom(\cF^*_{6,z},W)\to   \bigwedge^3\cF_{6,z}
$ which takes values in $\Ker R_z$ away from $E$, and this still remains true along $E$.
Hence, we have extended
the definition of the map  (\ref{www})  over $U_2$ as the fiber of a map
$$
f: \cH\!om(\cF_6^*,W\otimes \cO_{U_2})\to   \cK\hskip-1mm er R\vert_{U_2}
$$
between two vector bundles of rank 18.
 One checks that $R$ is surjective in codimension $1$.
It follows that the vanishing of $\det (f)$ gives us a divisor in the linear system
$$\vert\bigl(\det (\bigwedge^3\cF_6)\otimes (\det\cL_2)^{-1}\otimes( \det\cF_6)^{-3}\vert=\vert  (\det\cL_2)^{-1}\otimes( \det\cF_6)^7\vert.$$
Using (\ref{detf}) and the fact that, analogously, the determinant of $\cL_2$ pulls back to
$(\cO_S(1)\boxtimes \cO_S(1))(- \widetilde E)$ on $\widetilde{S\times S}$, we see that this line bundle   pulls back  to the line bundle
$$(\cO_S(1)\boxtimes \cO_S(1))^6(-20 \widetilde E)$$
on $\widetilde{S\times S}$,  and this concludes the proof of the lemma.
\end{proof}

We can conclude  the proof of Proposition \ref{propphi*H} with the following lemma.

\begin{lemm}\label{2lemme}
1) The map $\phi$ contracts $D'$ to $Y_\sigma^3$, so that
the rank of $P$ along $D'$ is $3$.

2) The divisor $D'$ has everywhere multiplicity $2$.
\end{lemm}

Indeed, as $P$ has rank $3$ along $D'$, the reduced divisor $D'_{\rm red}$ underlying $D'$   appears with multiplicity $3$ in the divisor defined by
$\det (P)$. By Lemmas \ref{1lemme} and \ref{2lemme}, $r^*D'_{r\rm ed}$ belongs to the linear system
$\vert (\cO_S(1)\boxtimes \cO_S(1))^3(-10 \widetilde E)\vert $. Hence we get, using (\ref{phi*}) and  (\ref{detf}):
\begin{eqnarray*}
r^*\phi^*\cO_{Y_\sigma}(1)&=&r^*\bigl((\det \cF_6)(3D'_{\rm red})\bigr)\\
&=&(\cO_S(1)\boxtimes \cO_S(1))(-3 \widetilde E)\otimes(\cO_S(1)\boxtimes \cO_S(1))^9(-30 \widetilde  E)\\
&=&
(\cO_S(1)\boxtimes \cO_S(1))^{10}(-33 \widetilde  E),
\end{eqnarray*}
which is the content of the proposition.
\end{proof}

\begin{proof}[Proof of Lemma \ref{2lemme}.]
1)  By definition of $D'$, a point $z$ in $U_1\cap U_2$ has the property
that the point  $\phi(z)$ of $ Y_\sigma$ corresponds to
a vector subspace $W_6\subset V_{10}$ such that
$p\vert_{W_6}:W_6\to V_{10}/W$ is not of maximal rank. In other words, with the notation of
Proposition  \ref{Yi}, $\phi(z)$ belongs to $Y^i_\sigma$, for some $i\geq 1$. Furthermore,  the rank of $P$ at $z$ is equal to the rank of $p\vert_{W_6}$. By Proposition
\ref{Yi}, we have $\dim Y^i_\sigma\leq 2$ for $i\geq 1$, thus 1) is equivalent to the following.

\begin{claim} If $\sigma$ is general (in the hypersurface parametrizing
singular $F_\sigma$), no divisor of $S^{[2]}$ is contracted to  $Y_\sigma^1$ or $Y_\sigma^2$.
\end{claim}

Let us first consider the case of $Y_\sigma^1$. On $U_2$,    the fiber of $\phi$ over a point $[W_6]\in Y_\sigma^1$ is contained
 in
the set of $([W'],[W''])\in S^{(2)}$ such that $W_6\subset p^{-1}( W'\oplus W'')$.
{As $[W_6]\in Y_\sigma^1$, the space   $p(W_6)$ has dimension $5$.} Let $p(W_6)^\bot\subset (V_{10}/W)^*$ be the 2-dimensional space of
linear forms vanishing on $p(W_6)$. As $p(W_6)$ has codimension $1$ in $W'\oplus W'' $, the space
$p(W_6)\cap W'$ has codimension $1$ in $W'$, and the rank at $[W']$ of the evaluation map
$$\ev :p(W_6)^\bot\otimes \cO_S\to \cE_3$$
is $1$. If the fiber $\phi^{-1}([W_6])$ has positive dimension, the set of $[W']$ as above contains  a curve,
and the saturation $(\Im \ev)_{\rm sat}$ has rank $2$ and nontrivial effective determinant.
 But $S$ is very general, hence its Picard group is cyclic, generated by
$\det\cE_3=\cO_S(1) $, so the curve above has to be  in a linear system $\vert\cO_S(l)\vert$, for some $l>0$. We get a contradiction from the fact that the
cokernel $\cE_3/(\Im\ev)_{\rm sat}$ is a rank-1 torsion-free sheaf with
determinant equal to $\cO_S(1-l)$, with $l\geq1$; this would imply that
$\cE_3^*(1-l)$ has a nonzero section for some $l\geq1$, which is absurd.

\medskip

We now turn to the case of $Y_\sigma^2$. A point $[W_6]$ in $ Y_\sigma^2$ is
 such that $W_2 :=W\cap W_6$ has dimension 2 and $W_4 := p(W_6)$ has dimension 4. We want to show that the
set of $([W'],[W''])\in S^{(2)}$ with $W_4\subset W'\oplus W''$
is  finite.

We  count parameters as in the proof of Proposition \ref{Yi}.2), whose notation we keep. We want to compute the dimension of the set of $([\sigma], [W],[W_1],[W_2],[W_4],[W'_3],[W'],[W''])$ such that $W=W_1\oplus W_2$, $V_{10}=W\oplus W_4\oplus W'_3$, $W_4\subset W'\oplus W''\subset W_4\oplus W'_3$, and, in addition to the conditions
\begin{equation}\label{cond1}
\sigma\vert_{\bigwedge^2W \wedge V_{10}}=\sigma\vert_{\bigwedge^3(W_2\oplus W_4)}=0
\end{equation}
  of that proof, such that
$$
\sigma\vert_{\bigwedge^3(W\oplus W')}=\sigma\vert_{\bigwedge^3(W\oplus W'')}=0.
$$
This means that the forms $\sigma_1$ and $\sigma_2$ must satisfy
     \begin{equation}\label{cond2}
\sigma_1\vert_{W\otimes \bigwedge^2W'}=\sigma_1\vert_{W\otimes \bigwedge^2W''}=\sigma_2\vert_{ \bigwedge^3W'}=\sigma_2\vert_{ \bigwedge^3W''}=0.
\end{equation}
Observe that we may assume   $\dim(W'\cap W_4)=\dim(W''\cap W_4)=1$, as
the case where one of these dimensions is $\geq 2$ can be ruled out by the method used in the   proof above. Then one checks   that the $9+9+1+1=20$ conditions (\ref{cond2}) are transverse to the conditions (\ref{cond1}).

Therefore, using the numbers from the proof of Proposition \ref{Yi}.2), there are $70+2+9+9=90$ parameters for the choice of $W_1$, $W_2$, $W_4$, $W'_3$, $W'$, and $W''$, and,  $81-20=61$ parameters for the choice of $[\sigma]$. It follows that the set of $([\sigma], [W],[W_6],[W'],[W''])$ such that $F_\sigma$ is singular at $[W]$ and the point $([W'],[W''])$ of $S^{[2]}$ is mapped to the point $[W_6]$ of $ Y^2_\sigma$ has the same dimension   $120$ as the set of $([\sigma], [W],[W_6])$. It follows that the corresponding projection is generically finite, which
proves the claim.

\medskip

2) By the proof of 1), we now have another set-theoretic description of the divisor $D'$: on $U_2$, it
is the set of pairs $([W'],[W''])\in S^{[2]}$ such that there exists
 $[W_6]\in Y_\sigma^3$ with
$$W\subset W_6\subset W\oplus W'\oplus W''.$$
This locus has another determinantal description as follows.
A $W_6$ as above is determined by
its $3$-dimensional  projection $W_3$ in $ W'\oplus W''$, and $[W_3]$ must be an element of
$S$. Write $W_3$ as the graph of a map
$$v:W'\rightarrow W''$$
(the nontransverse cases cannot fill in a divisor, by arguments as above).

Recall that $S$ is defined by a 3-dimensional space of $2$-forms
$\sigma_1 \in W^*\otimes \bigwedge^2(V_{10}/W)^*$ and a $3$-form
$\sigma_2\in  \bigwedge^3(V_{10}/W)^*$.

Since $\sigma_1$ vanishes on $ W\otimes \bigwedge^2W'$ and $W\otimes \bigwedge^2W''$,  its restriction to
$W\otimes (W'\oplus W'')$ belongs to
$W^*\otimes W^{'*}\otimes W^{''*}$, so that
the vanishing of
$(\Id,v)^*\sigma_1$ provides $9$ linear equations on $v$.
The existence of a nonzero solution $v$ is thus equivalent to
the nonindependence of these linear equations.
We have a morphism (only   defined on $U_2\moins  E $, but globally defined on
the double cover $\widetilde{S\times S}$)
$$
\begin{array}{rcl}
\beta:\cH\! om (\cS_1,\cS_2)&\to & W^*\otimes\bigwedge^2p_1^*\cE_3\\
v&\mapsto& (\Id,v)^*\sigma_1.
\end{array}
$$
At a point $([W'],[W''])$ where
$\beta$ does not have maximal rank, $\beta_{W',W''}^{-1}(0)$ contains a line $V_1$, and
we now have to impose a supplementary condition on
$v\in V_1$ in order that the corresponding $W_3=\Im(\Id,v)$ be
in $S$, namely
$$(\Id,v)^*\sigma_2=0.$$
Observe that this last equation is quadratic (inhomogeneous) in $v$, and vanishes at $v=0$.
Hence there is in fact  a unique $W_3\subset W'\oplus W''$ for a general point
$([W'],[W''])$ in the divisor $D''$ defined by $\det(\beta)$.

In order to conclude, we have to prove the following.

\begin{claim}
For general $\sigma$, the divisor $D''$ is reduced and $D'=2D''$.
\end{claim}

The first fact is elementary and left to the reader. As for the second one, it follows from
the observation that at a point
$([W'],[W''])$ of $ S^{[2]}$, the linear part
$$
\begin{array}{rcl}
\vec f_{W',W''}: \Hom(W'\oplus W'',W)&\to&  \big(\bigwedge^2W'\otimes W''\bigr)^*\oplus \bigl(W'\otimes \bigwedge^2W''\bigr)^*\\
u&\mapsto& (\Id,u)^*\sigma_1,
\end{array}
$$
 of the morphism  (\ref{www})
is nothing but the transpose of the direct sum of
the morphism
$$
\begin{array}{rcl}
\beta_{W',W''}:   \Hom(W',W'')&\to&\bigl( W\otimes\bigwedge^2W'\bigr)^*\\
v&\mapsto& (\Id,v)^*\sigma_1
\end{array}
$$
introduced above, and its   counterpart
$$\beta_{W'',W'}: \Hom(W'',W')\rightarrow \bigl( W\otimes\bigwedge^2W''\bigr)^*,
$$
obtained by exchanging $W'$ and $W''$ (here we identify
${W'}^*$ with $\bigwedge^2W'$ and similarly for $W''$). It follows from our discussion
above that $\det(\beta_{W',W''})$ and $\det(\beta_{W'',W'})$ both vanish simply along $D'_{\rm red}$. Hence
$\det(\vec f_{W',W''})$ vanishes with multiplicity $2$ along $D'_{\rm red}$.
\end{proof}

\section{The fourfold $Y_\sigma$ as a deformation of  $\Hilb^2(K3)$}

It follows from Theorem \ref{depart} and \cite{guan} that the Hodge numbers of $Y_\sigma$ are the same as those of the Hilbert scheme of pairs of points on a K3 surface. We prove in this section
  the following more precise result.

\begin{theo} \label{theodef} The variety $Y_\sigma$ with its Pl\"ucker polarization is a deformation of
$(S^{[2]},L)$, where $S$ is the  K3 surface of genus 12 introduced in the previous section,
and $L$ is the line bundle  on $S^{[2]}$ whose pull-back to $\widetilde{S\times S}$ is $(\cO_S(1)\boxtimes \cO_S(1))^{10}(-33\widetilde E)$.
\end{theo}

We want to use the degeneration described in the previous section, but we have to be careful, as
the central fiber is very singular and only birationally equivalent to
$S^{[2]}$. We will borrow part of the arguments of \cite{huy}.

The proof of Theorem \ref{theodef} will follow from a computation of Hilbert polynomials and from the following
variant of Huybrechts' theorem saying that birationally equivalent hyper-K\"ahler manifolds are deformation equivalent.

We start from the following more general situation: $X$ is  an irreducible hyper-K\"ahler manifold of dimension $2n$, $Y$ is a  normal  projective variety, and $\phi:X\dra Y$ is a birational map. We will assume
that $Y$ is a projective degeneration of irreducible hyper-K\"ahler
manifolds, which means that there is    an ample line bundle $H$ on $Y$ and a flat projective family
$$(\cY,\cH)\to \Delta,\quad{\rm where}\quad\cH\in \Pic(\cY),$$
with central fiber $(Y,H)$ and with general fiber $(Y_t,H_t)$, with $H_t$ ample on $Y_t$, an irreducible hyper-K\"ahler manifold. Note that this implies in particular that the canonical bundle of $Y$ is trivial on its smooth locus $Y_{\rm reg}$.

\begin{prop}\label{varianthuy}   Assume that  the line bundle
$L:=\phi^*H$ on $X$ has the following property:
\begin{equation}\label{egalhilbert}
\forall k\in \Z\quad  \chi(X,L^k)=\chi(Y,H^k).
\end{equation}
Then a small deformation of $(X,L)$ is isomorphic to a (smooth) deformation of $(Y,H)$.
\end{prop}

\begin{proof}
Let
$T\subset Y$ be the union of the singular locus
of $Y$ and the indeterminacy locus
of $\phi^{-1}$, and let $D\subset X$ be
the union of $\phi^{-1}(T)$ and of the indeterminacy locus of $\phi$.
(Note that in the case where $Y$ is not smooth, $D$ may have divisorial components.)

\begin{lemm}\label{ind} The map $\phi$  induces an isomorphism
$$X\moins  D\isom Y\moins  T.$$
\end{lemm}

\begin{proof}
By construction, $\phi$ induces a morphism $ \phi_D:X\moins  D\to Y\moins  T\subset Y_{\rm reg}$, and $\phi^{-1}$   a morphism $\phi^{-1}_T:  Y\moins  T\to X$. Let $\eta_X$ be a generator of the (1-dimensional) space of holomorphic $2$-forms on $X$ and let $\eta_Y$ be its pull-back by $\phi^{-1}_T$. It is a nonzero holomorphic $2$-form on $Y\moins  T$. The form
$\eta_Y^n$ is nonzero on $Y\moins  T$, hence it does not vanish there, because the canonical bundle of $Y\moins  T$ is trivial
and $Y\moins T$ has no nonconstant holomorphic functions.
In other words, $\eta_Y$ is nondegenerate on $Y\moins  T$. Since $(\phi^{-1}_T)^*(\eta_X)=\eta_Y$, we conclude  that $\phi^{-1}_T$ is \'{e}tale.

Let us show that $\phi_D$ is surjective.
Let $y\in Y\moins T$; then $\phi^{-1}$ is defined
at $y$, and $\phi^{-1}$ is \'{e}tale at $y$. It follows that $\phi$ is defined at $\phi^{-1}(y)$ and
$y=\phi(\phi^{-1}(y))$. As $y\notin T$ and $\phi$
is defined at $\phi^{-1}(y)$, we conclude   $\phi^{-1}(y)\notin D$.

Finally, we have $\phi_D^*(\eta_Y)=\eta_X\vert_{X\moins D}$; in particular, since $\eta_X$ is nondegenerate, we obtain as above that $\phi_D$ is \'{e}tale.

 Hence we have proved that $\phi_D$ is an \'{e}tale surjective birational  morphism between
the smooth varieties $X\moins  D$ and $Y\moins  T$. It is therefore an isomorphism
and the lemma is proved.
\end{proof}

\begin{lemm}\label{lemmeh0} Under the assumptions of Proposition
\ref{varianthuy}, for all $k\ge 0$, the map $\phi^*$ induces an isomorphism
$$H^0(Y,H^k)\isom H^0(X,L^k).$$
\end{lemm}

\begin{proof} Consider the following
composition of maps:
$$H^0(Y,H^k)\to H^0(Y\moins  T,H^k)\stackrel{\phi^*}{\to} H^0(X\moins  D,L^k).$$
The first map is bijective by the normality of $Y$. The second map is an isomorphism by Lemma \ref{ind}.
It follows that we get an isomorphism
$$H^0(Y,H^k)\to H^0(X\moins  D,L^k)$$
which obviously factors as
$$H^0(Y,H^k)\stackrel{\phi^*}{\lra} H^0(X,L^k)\lra H^0(X\moins  D,L^k),$$
where the last map is the restriction map, which is injective.
Hence the map
$\phi^*:H^0(Y,H^k){\to} H^0(X,L^k)$ is also bijective.
\end{proof}

The proof of Proposition \ref{varianthuy} is now immediate.
Indeed, consider a deformation $\pi:(\cX,\cL)\to \Delta $ of the pair $(X,L)$, such that
for a general point  $t\in \Delta$, the group $\Pic (X_t)$ has rank $1$.

 We claim that the line bundle $L_t$ is ample on $X_t$. (This is the only place where
 we will use the fact that $(Y,H)$ is a projective degeneration
 of an irreducible hyper-K\"ahler manifold.)
  Indeed, its Hilbert polynomial
equals the Hilbert polynomial of $L$ on $X$, hence of $H$ on $Y$ by our main assumption, or equivalently of $H_s$ on
 $Y_s$ for general $s$. Its  terms of degree $2n$  and
$2n-2$ are therefore equal to those of
$H_s$ on $Y_s$. Of course, the  terms of degree $2n$ are  positive multiples of
$q_X(L)^n$ and $q_{Y_s}(H_s)^n$ respectively,
where $q_X$ is the Beauville-Bogomolov quadratic form on $H^2(X,\Z)$ (\cite{be}) and
similarly for $q_{Y_s}$. Next, the terms of degree $2n-2$ are multiples of
$q_X(L)^{n-1}$ and $q_{Y_s}(H_s)^{n-1}$, the signs of the coefficients being the same.
This indeed follows from Riemann-Roch formula and the fact
(which we can apply to $X$ and $Y_s$) that for any $2n$-dimensional
irreducible hyper-K\"ahler
manifold $Z$, and any degree-2 class $\alpha$ on $Z$,
$$q_Z(\alpha)^{n-1}=\mu_Z c_2(T_X)\alpha^{2n-2},$$
with $\mu_Z>0$.
In conclusion, $q_X(L)^n$ and $q_{Y_s}(H_s)^n$ have the same sign (and are nonzero), and so do
$q_X(L)^{n-1}$ and $q_{Y_s}(H_s)^{n-1}$. Hence $q_X(L)$ and $q_{Y_s}(H_s)$ have the
same sign. As $q_{Y_s}(H_s)>0$, we get $q_X(L)>0$.
 By
\cite{huy}, this implies now that $X_t$ is projective, and, as $\Pic( X_t)$ is cyclic, either $L_t$ or $L^{-1}_t$ is ample. The second case is impossible
because $H^0(X,L^k)=0$ for $k<0$, hence by semi-continuity, $H^0(X_t,L^k_t)=0$ for $k<0$. Thus the claim is proved.

But then, we have
$$\forall k>0\quad \chi(X,L^k)=\chi(X_t,L^k_t)=h^0(X_t,L^k_t).$$
On the other hand, we have by Lemma \ref{lemmeh0}
$$ h^0(X,L^k)=h^0(Y,H^k)=\chi(Y,L^k)$$
for $k$ large enough, and the last term equals by assumption $\chi(X,L^k)$.
Hence we get
$$\forall k\gg 0  \quad h^0(X,L^k)=h^0(X_t,L^k_t),$$
and it follows by the semi-continuity and  base change theorems that
the locally free sheaf
$\pi_*(\cL^k)$
on $\Delta$ has for fiber $H^0(X,L^k)$ at $0$ and of course $H^0(X_t,L^k_t)$ at $t$.
But then, we get a flat projective family $\cY$ over $\Delta$
by the formula
$$\cY=\cP roj\Bigl( \bigoplus_{k\ge 0} \pi_*(\cL^k)\Bigr).$$
By the above base change result and Lemma \ref{lemmeh0}, the fiber of this family over $0$
is isomorphic to $Y$, endowed with the line bundle $H$, while the fiber over $t$ is $X_t$ endowed with the line bundle
$L_t$.
\end{proof}

Theorem \ref{theodef} will be obtained as a consequence of Proposition
\ref{varianthuy}, applied
to  the birational map constructed in the previous section between
$X=S^{[2]}$ and $Y=Y_\sigma$, with $H=\cO_{Y_\sigma}(1)$ and
$L=\phi^*\cO_{Y_\sigma}(1)$.

As we know that the singular variety $Y_\sigma$ is normal (Proposition \ref{Yi}) and   is  a
projective degeneration of an irreducible hyper-K\"ahler fourfold by Theorem
\ref{depart}, in  order to apply Proposition \ref{varianthuy}, we only need to check the assumptions concerning  the Hilbert polynomials, and this is done
in the following lemma.

\begin{lemm}\label{2fin} The Hilbert polynomials of $\cO_{Y_\sigma}(1)$ and  $L$ coincide.
\end{lemm}

\begin{proof}
 Let us first compute the Hilbert polynomial of $\cO_{Y_\sigma}(1)$.
 We claim that for any integer  $k$,
$$
\chi(Y_\sigma,\cO_{Y_\sigma}(k))=3+\frac{55 }{2}k^2 +\frac{121 }{2}k^4.
$$
Indeed, the Hilbert polynomial is given by the Riemann-Roch formula. Let us denote by $c_i$ the Chern classes
of the vector bundle $\cS_6^*\vert_{Y_\sigma}$, so in particular $c_1=c_1(\cO_{Y_\sigma}(1))$. Recalling that the class
of $Y_\sigma$ in $G(6,V_{10})$ is $c_{20}(\bigwedge^3\cE_6)$, Macaulay gives us the following
intersection numbers on $Y_\sigma$:
\begin{equation}\label{lesnombres}
c_1c_3=330,\ c_4=105,\ c_1^2c_2=825,\ c_2^2=477,\ c_1^4=1452.
\end{equation}
 As $Y_\sigma$ is a hyper-K\"ahler variety,
its odd-degree Chern classes $c_1(T_{Y_\sigma})$ and $c_3(T_{Y_\sigma})$ vanish.
Hence the Riemann-Roch formula takes the following very simple form:
\begin{equation}\label{rrhyp}
\chi(Y_\sigma,\cO_{Y_\sigma}(k))=\chi(Y_\sigma,\cO_{Y_\sigma})+
\frac{c_2(T_{Y_\sigma})c_1^2}{24}k^2+
\frac{c_1^4}{24}k^4.
\end{equation}
The first term of the sum equals $3$ by Theorem \ref{depart}. According to  (\ref{lesnombres}), the last term equals $\frac{121 }{2}k^4$. For the middle term, we need to compute $c_2(T_{Y_\sigma})c_1^2$.
 This is a tedious but straightforward computation. The tangent bundle $T_{Y_\sigma}$ appears in the normal bundle sequence:
$$0\to T_{Y_\sigma}\to T_{G(6,V_{10})}\vert_{Y_\sigma}\to  \bigwedge^3\cE_6\vert_{ Y_\sigma}\to0.$$
Using the equality $T_{G(6,V_{10})}=\cE_6\otimes \bigl( (V_{10}\otimes\cO_{G(6,V_{10})})/\cS_6\bigr)$,
we now compute
$$c_2(T_{Y_\sigma})=5c_1^2-8c_2.$$
which together with (\ref{lesnombres}) gives
$$c_2(T_{Y_\sigma})c_1^2= 660.  $$
Thus the claim is proved.

We now turn to the computation of the Hilbert polynomial of the line bundle $L$ on
$S^{[2]}$.

This is an  explicit and standard computation. As $S^{[2]}$ is hyper-K\"ahler, formula (\ref{rrhyp}) applies as well to $S^{[2]}$ and $L$:
$$\chi(S^{[2]},L^k)=\chi(S^{[2]},\cO_{S^{[2]}})+
\frac{c_2(T_{S^{[2]}})c_1(L)^2}{24}k^2+
\frac{c_1(L)^4}{24}k^4.$$
The first number in the sum is $3$. It thus suffices to show the equalities
$$c_1(L)^4=1452\quad{\rm and}\quad c_2(T_{S^{[2]}})c_1(L)^2= 660.$$
By Proposition \ref{propphi*H},
the pull-back of $L$  on $\widetilde{S\times S}$ is
 $$(\cO_S(1)\boxtimes \cO_S(1))^{10}(-33 \widetilde E).$$
Letting
$$\ell_i=p_i^*c_1(\cO_S(1))\quad{\rm and}\quad e=[\widetilde  E]$$
on $\widetilde{S\times S}$, we need to show
\begin{eqnarray}
(10( \ell_1+ \ell_2)-33e)^4&=&2904\nonumber\\
(10( \ell_1+ \ell_2)-33e)\cdot r^*c_2(T_{S^{[2]}})&=&1320.\label{e1}
\end{eqnarray}
The first equality follows from
 \begin{equation}\label{dautresnombres}
 \ell_i^2=22,\ \ell_i^3=0,\  \ell_i\ell_je^2=-22,\ e^4=24,
 \end{equation}
 together with the vanishing of the contributions of any odd power of $e$.

 As to the second equality, note the two exact sequences, which compare
 $r^*\Omega_{S^{[2]}}$ and $\eps^*\Omega_{S\times S}$:
 $$0\to  r^*\Omega_{S^{[2]}}\to \Omega_{\widetilde{S\times S}}\to \cO_{\widetilde E}(-\widetilde E)\to 0,$$
 $$0\to  \eps^*\Omega_{S\times S}\to \Omega_{\widetilde{S\times S}}\to\cO_{\widetilde E}(2 \widetilde E)\to 0.$$
 This gives us the following formula for the total Chern class
 of $r^*\Omega_{S^{[2]}}$:
 \begin{equation}
 \label{formulercotagt} r^*c(\Omega_{S^{[2]}})=\eps^*c(\Omega_{S\times S})c(\cO_{\widetilde E}(2 \widetilde E)) c(\cO_{\widetilde E}(-\widetilde E))^{-1}.
 \end{equation}
 For $i\in\{1,2\}$, let  $o_i$ be the class of a fiber of  the projection
  $p_i:\widetilde{S\times S}\to S$; we have
 $$\eps^*c(\Omega_{S\times S})=(1+24 o_1)(1+24 o_2)$$
 and we deduce from (\ref{formulercotagt})
 $$r^*c_2(\Omega_{S^{[2]}})=24o_1+24o_2-3e^2.$$
Equality (\ref{e1}) then
 follows from
 (\ref{dautresnombres}) together with
 $o_1\ell_2^2=o_2\ell_1^2=22$, $o_i\ell_i=0$, and $o_ie^2=-1$.
\end{proof}
At this point, we have shown that a small smooth deformation of
$(Y_\sigma,\cO_{Y_\sigma}(1))$ is isomorphic to a small deformation of $(S^{[2]},L)$. In order to conclude the proof
of  Theorem \ref{theodef}, it only remains to prove  the following lemma (and use Proposition \ref{Yi}.4)).

\begin{lemm} Whenever $Y_\sigma$ has dimension $4$, any small deformation of $(Y_\sigma,\cO_{Y_\sigma}(1))$ is given by a deformation of
$\sigma$.
\end{lemm}

\begin{proof} Let $Z$ be a local complete intersection projective scheme and let $L$ be a line bundle on $Z$. 
The first Chern class of $L$, seen as an element of $H^1( Z, \Omega_Z)$,
defines an extension
 \begin{equation}\label{ext}
 0\to \Omega_Z\to \cP_{Z,L}\to \cO_Z\to 0
  \end{equation}
 and first-order deformations of the pair $(Z,L)$ are parametrized by $\Ext^1_Z (  \cP_{Z,L}, \cO_Z )$.

In our situation,  $Y_\sigma$ is the zero-set of the section $\sigma$ of the vector bundle $\cF=\bigwedge^3\cE_6$ on   $G:=G(6,V_{10})$. The discussion above applies to both $(G,\cO_G(1))$ and $(Y_\sigma,\cO_{Y_\sigma}(1))$. Since the normal bundle to $Y_\sigma$ in $G$ is $\cF\vert_{Y_\sigma}$, we obtain an exact sequence
$$0\to \cF^*\vert_{Y_\sigma}\to \cP_{G,\cO_G(1)}\vert_{Y_\sigma}\to\cP_{Y_\sigma,\cO_{Y_\sigma}(1)}\to 0.
$$
from which we deduce an exact sequence
 \begin{equation}\label{ccc}
 H^0({Y_\sigma},\cF\vert_{Y_\sigma})\stackrel{\beta}{\to} \Ext^1_{Y_\sigma} (  \cP_{Y_\sigma,\cO_{Y_\sigma}(1)}, \cO_{Y_\sigma} )
\to
\Ext^1_{Y_\sigma} ( \cP_{G,\cO_G(1)}\vert_{Y_\sigma}, \cO_{Y_\sigma} ).
  \end{equation}
 We need to show that the composition
$$ H^0(G,\cF)\stackrel{\alpha}{\lra} H^0({Y_\sigma},\cF\vert_{Y_\sigma})\stackrel{\beta}{\lra} \Ext^1_{Y_\sigma} (  \cP_{Y_\sigma,\cO_{Y_\sigma}(1)}, \cO_{Y_\sigma} )
$$
is surjective. We will prove that both $\alpha$ and $\beta$ are surjective.

Using the Koszul resolution for $\cO_{Y_\sigma}$, we see that $\alpha$  is surjective if
$$ H^i(G ,\cF\otimes \bigwedge^i\cF^*)=0
$$
for all $i>0$, a fact that can be checked using Bott's theorem and the program LiE, as explained in Remark \ref{manhan}.

To show the surjectivity of $\beta$, it is enough by (\ref{ccc}) to show   that
$\Ext^1_{Y_\sigma} (  \cP_{G,\cO_G(1)}\vert_{Y_\sigma}, \cO_{Y_\sigma} ) $ vanishes.
Consider the exact sequence
$$
H^1({Y_\sigma}, \cO_{Y_\sigma} )
\to \Ext^1_{Y_\sigma} (  \cP_{G,\cO_G(1)}\vert_{Y_\sigma}, \cO_{Y_\sigma} )
\to
H^1(Y_\sigma,T_G\vert_{Y_\sigma})
\stackrel{\gamma}{\to}
H^2({Y_\sigma}, \cO_{Y_\sigma} )$$
obtained from (\ref{ext}). Again, as in Remark \ref{manhan}, one shows using the Koszul resolution and Bott's theorem that $H^1({Y_\sigma}, \cO_{Y_\sigma} )$ vanishes. The map
$$\gamma:H^1(Y_\sigma,T_G\vert_{Y_\sigma})\to H^2({Y_\sigma}, \cO_{Y_\sigma} )
$$
is given by cup-product with $c_1(\cO_{Y_\sigma} (1))$. Using the Koszul resolution again, it is injective if  the cup-product maps
$$
\gamma_i : H^{i+1}(G,T_G\otimes \bigwedge^i\cF^* )
\to H^{i+2}(G, \bigwedge^i\cF^* )
$$
by $c_1(\cO_G(1)) $ are injective for all $i\ge 0$. The tangent bundle $T_G$ is isomorphic to $\cQ_4\otimes\cS_6^*$, hence appears in the exact sequence
 \begin{equation}\label{c3}
0\to \cS_6\otimes\cS_6^*
\to V_{10}\otimes \cS_6^*
\to T_G
\to 0,
\end{equation}
whose extension class is
 $$
 c_1(\cO_G(1))\in H^1(G,\Omega_G)\isom  \Ext^1_G(T_G,\cO_G)\subset \Ext^1_G(T_G,\cS_6\otimes\cS_6^*).
$$
Proceeding as above, one can show $H^{i+1}(G,\cS_6^*\otimes \bigwedge^i\cF^* )=0$ for all $i\ge 0$. In the long exact sequence in cohomology associated with (\ref{c3}), we deduce that  the edge map
$$
H^{i+1}(G,T_G\otimes \bigwedge^i\cF^* )
\stackrel{\gamma_i}{\to} H^{i+2}(G, \bigwedge^i\cF^* )\hookrightarrow  H^{i+2}(G,\cS_6\otimes\cS_6^*\otimes  \bigwedge^i\cF^* )
$$
is bijective. This proves this injectivity of  $\gamma_i$, hence the lemma.
\end{proof}

\section{Further comments and questions\label{section5}}

The geometric invariant theory of the $3$-vectors $\sigma\in\bigwedge^3V_{10}^*$
does not seem to have been studied.
We introduced in section \ref{section1} a natural hypersurface in the moduli space
$$\P(\bigwedge^3V_{10}^*)\,/\!/\PGL(V_{10}^*).$$
It parametrizes those $Y_\sigma$ containing a line in the Pl\"ucker embedding.
Section \ref{section2} was devoted to another hypersurface in this moduli space, parametrizing singular $F_\sigma$.

There is a third natural hypersurface in this moduli space:
it is the set of $\sigma$ for which  $F_\sigma$ contains
a $10$-dimensional Grassmannian $G(2,7)\subset G(3,V_{10})$. Here we
choose a $V_8\subset V_{10}$ together with a nonzero $x$ in $ V_8$, and
we see $G(2,7)$ as the set of $W_3\subset V_{10}$ such that $x\in W_3\subset V_8$.
The fact that this $G(2,7)$ is contained in $F_\sigma$ is equivalent
to the fact that the $2$-form $\Int_x\sigma$ vanishes on $V_8$. That the existence of such
a subvariety of $F_\sigma$ is a divisorial condition on $\sigma$ follows from the equality $h^{9,11}(F_\sigma)=1$ and the semi-regularity of the
embedding $G(2,7)\subset F_\sigma$, which tells us that deforming $F_\sigma$
preserving $G(2,7)$ is equivalent to deforming $F_\sigma$ preserving the Hodge class
$[G(2,7)]$ (see \cite{bloch}).

A related question concerns the existence of a hypersurface in the moduli space where
$Y_\sigma$ is actually isomorphic to $S^{[2]}$ for some $K3$ surface $S$.
This should hold along a hypersurface where the Picard number of $Y_\sigma$ jumps (or equivalently,
by Corollary \ref{cor6avril}, where the dimension of the space of degree-$20$ Hodge classes on $F_\sigma$ jumps).

There are two  families of $K3$ surfaces which are natural candidates, namely those of genus
$16$ and those of genus $21$. Indeed, the first ones admit a rigid rank-$2$ vector bundle
with $10$ independent sections, so that their second Hilbert schemes carry a rigid rank-$4$ vector bundle
with $10$ independent sections, which embeds them into $G(4,10)\isom G(6,10)$. Similarly,
$K3$ surfaces of genus $21$ admit a rigid rank-$3$ vector bundle with $10$ independent sections, so that
their second Hilbert schemes carry a rigid rank-$6$ vector bundle
with $10$ independent  sections, which embeds them into $G(6,10)$.
In both cases and surprisingly enough, the degree of the hyper-K\"ahler subvarieties of $G(6,10)$
that one obtains is $1452$, which is the degree of $Y_\sigma$.
However the other Chern numbers of the tautological vector bundle (see (\ref{lesnombres}))
do not coincide.

\end{document}